\documentclass[11pt]{article}
\usepackage{amssymb,amsmath}
\usepackage[mathscr]{eucal}
\usepackage[cm]{fullpage}
\usepackage[english]{babel}
\usepackage[latin1]{inputenc}
\def\dom{\mathop{\mathrm{Dom}}\nolimits}
\def\im{\mathop{\mathrm{Im}}\nolimits}
\def\d{\mathrm{d}}
\def\id{\mathrm{id}}
\def\N{\mathbb N}
\def\PT{\mathcal{PT}}
\def\T{\mathcal{T}}
\def\Sym{\mathcal{S}}
\def\DP{\mathcal{DP}}
\def\A{\mathcal{A}}
\def\B{\mathcal{B}}
\def\C{\mathcal{C}}
\def\D{\mathcal{D}}
\def\DPS{\mathcal{DPS}}
\def\DPC{\mathcal{DPC}}
\def\ODP{\mathcal{ODP}}

\def\POD{\mathcal{POD}}
\def\POR{\mathcal{POR}}
\def\I{\mathcal{I}} 
\def\ro{{\hspace{.2em}}\rho{\hspace{.2em}}}
\newtheorem{theorem}{Theorem}[section]
\newtheorem{proposition}[theorem]{Proposition}
\newtheorem{corollary}[theorem]{Corollary}
\newtheorem{lemma}[theorem]{Lemma}

\newenvironment{proof}{\begin{trivlist}\item[\hskip%
\labelsep{\bf Proof.}]}%
{\qed\rm\end{trivlist}}
\newcommand{\qed}{{\unskip\nobreak
\hfil\penalty50\hskip .001pt \hbox{}
          \nobreak\hfil
         \vrule height 1.2ex width 1.1ex depth -.1ex
           \parfillskip=0pt\finalhyphendemerits=0\medbreak}}

\newcommand{\lastpage}{\addresss}

\newcommand{\addresss}{\small \sf  
\noindent{\sc V\'\i tor H. Fernandes}, 
Center for Mathematics and Applications (CMA), 
FCT NOVA and Department of Mathematics, FCT NOVA, 
Faculdade de Ci\^encias e Tecnologia, 
Universidade Nova de Lisboa, 
Monte da Caparica, 
2829-516 Caparica, 
Portugal; 
e-mail: vhf@fct.unl.pt. 

\medskip

\noindent{\sc T\^ania Paulista}, 
Departamento de Matem\'atica, 
Faculdade de Ci\^encias e Tecnologia, 
Universidade NOVA de Lisboa, 
Monte da Caparica, 
2829-516 Caparica, 
Portugal; 
e-mail: t.paulista@campus.fct.unl.pt. 
}

\title{On the monoid of partial isometries of a cycle graph}

\author{V\'\i tor H. Fernandes\footnote{This work is funded by national funds through the FCT - Funda\c c\~ao para a Ci\^encia e a Tecnologia, I.P., under the scope of the projects UIDB/00297/2020 and UIDP/00297/2020 (Center for Mathematics and Applications).}~
and T\^ania Paulista 
}

\begin{document}

\maketitle

\begin{abstract}
In this paper we consider the monoid $\DPC_n$ of all partial isometries of a $n$-cycle graph $C_n$. 
We show that $\DPC_n$ is the submonoid of the monoid of all oriented partial permutations on a $n$-chain whose elements are precisely all restrictions of the dihedral group of order $2n$.  
Our main aim is to exhibit a presentation of $\DPC_n$. 
We also describe Green's relations of $\DPC_n$ and calculate its cardinal and rank.  
\end{abstract}

\medskip

\noindent{\small 2020 \it Mathematics subject classification: \rm 20M20, 20M05, 05C12, 05C25.} 

\noindent{\small\it Keywords: \rm transformations, orientation, partial isometries, cycle graphs, rank, presentations.} 

\section*{Introduction}\label{presection} 

Let $\Omega$ be a finite set. As usual, let us denote by $\PT(\Omega)$ the monoid (under composition) of all 
partial transformations on $\Omega$, by $\T(\Omega)$ the submonoid of $\PT(\Omega)$ of all 
full transformations on $\Omega$, by $\I(\Omega)$ 
the \textit{symmetric inverse monoid} on $\Omega$, i.e. 
the inverse submonoid of $\PT(\Omega)$ of all 
partial permutations on $\Omega$, 
and by $\Sym(\Omega)$ the \textit{symmetric group} on $\Omega$, 
i.e. the subgroup of $\PT(\Omega)$ of all 
permutations on $\Omega$. 

\smallskip 

Recall that the \textit{rank} of a (finite) monoid $M$ is the minimum size of a generating set of $M$, i.e. 
the minimum of the set $\{|X|\mid \mbox{$X\subseteq M$ and $X$ generates $M$}\}$. 

Let $\Omega$ be a finite set
with at least $3$ elements.
It is well-known that $\Sym(\Omega)$
has rank $2$ (as a semigroup, a monoid or a group) and 
$\T(\Omega)$, $\I(\Omega)$ and $\PT(\Omega)$ have
ranks $3$, $3$ and $4$, respectively.
The survey \cite{Fernandes:2002survey} presents 
these results and similar ones for other classes of transformation monoids,
in particular, for monoids of order-preserving transformations and
for some of their extensions. 
For example, the rank of the extensively studied monoid of all order-preserving transformations of a $n$-chain is $n$,  
which was proved by Gomes and Howie \cite{Gomes&Howie:1992} in 1992. 
More recently, for instance, the papers 
\cite{
Araujo&al:2015,
Fernandes&al:2014,
Fernandes&al:2019,
Fernandes&Quinteiro:2014,
Fernandes&Sanwong:2014} 
are dedicated to the computation of the ranks of certain classes of transformation semigroups or monoids.

\smallskip 

A \textit{monoid presentation} is an ordered pair 
$\langle A\mid R\rangle$, where $A$ is a set, often called an \textit{alphabet}, 
and $R\subseteq A^*\times A^*$ is a set of relations of 
the free monoid $A^*$ generated by $A$. 
A monoid $M$ is said to be 
\textit{defined by a presentation} $\langle A\mid R\rangle$ if $M$ is
isomorphic to $A^*/\rho_R$, where $\rho_R$ denotes the smallest
congruence on $A^*$ containing $R$. 

Given a finite monoid, it is clear that we can always exhibit
a presentation for it, at worst by enumerating all elements from its multiplication table,
but clearly this is of no interest, in general. So, by determining a
presentation for a finite monoid, we mean to find in some sense a
\textit{nice} presentation (e.g. with a small number of generators and
relations).

A presentation for the symmetric group $\Sym(\Omega)$ was determined by Moore \cite{Moore:1897} over a century ago (1897). 
For the full transformation monoid $\T(\Omega)$, a presentation  
was given in 1958 by A\u{\i}zen\v{s}tat \cite{Aizenstat:1958} in terms of a certain 
type of two generator presentation for the symmetric group $\Sym(\Omega)$, 
plus an extra generator and seven more relations. 
Presentations for the partial transformation monoid $\PT(\Omega)$ 
and for the symmetric inverse monoid $\I(\Omega)$
were found by Popova \cite{Popova:1961} in 1961. 
In 1962, A\u{\i}zen\v{s}tat \cite{Aizenstat:1962} and Popova \cite{Popova:1962} exhibited presentations for the monoids of 
all order-preserving transformations and of all order-preserving partial transformations of a finite chain, respectively, and from the sixties until our days several authors obtained presentations for many classes of monoids. 
See also \cite{Ruskuc:1995}, the survey \cite{Fernandes:2002survey} and, 
for example, 
\cite{Cicalo&al:2015,
East:2011, 
Feng&al:2019,
Fernandes:2000, 
Fernandes:2001, 
Fernandes&Gomes&Jesus:2004, 
Fernandes&Quinteiro:2016, 
Howie&Ruskuc:1995}. 

\medskip 

Now, let $G=(V,E)$ be a finite simple connected graph. 

The (\textit{geodesic}) \textit{distance} between two vertices $x$ and $y$ of $G$, denoted by $\d_G(x,y)$, is the length of a shortest path between $x$ and $y$, i.e. the number of edges in a shortest path between $x$ and $y$. 

Let $\alpha\in\PT(V)$. We say that $\alpha$ is a \textit{partial isometry} or \textit{distance preserving partial transformation} of $G$ if 
$$
\d_G(x\alpha,y\alpha) = \d_G(x,y) ,
$$
for all $x,y\in\dom(\alpha)$. Denote by $\DP(G)$ the subset of $\PT(V)$ of all partial isometries of $G$. Clearly, $\DP(G)$ is a submonoid of $\PT(V)$. Moreover, as a consequence of the property 
$$
\d_G(x,y)=0 \quad \text{if and only if} \quad x=y, 
$$
for all $x,y\in V$, it immediately follows that $\DP(G)\subseteq\I(V)$. Furthermore, $\DP(G)$ is an inverse submonoid of $\I(V)$ 
(see \cite{Fernandes&Paulista:2022arxiv}). 

\smallskip 

Observe that, if $G=(V,E)$ is a complete graph, i.e. $E=\{\{x,y\}\mid x,y\in V, x\neq y\}$, then $\DP(G)=\I(V)$. 

On the other hand, for $n\in\N$, consider the undirected path $P_n$ with $n$ vertices, i.e. 
$$
P_n=\left(\{1,\ldots,n\},\{\{i,i+1\}\mid i=1,\ldots,n-1\}\right).
$$
Then, obviously, $\DP(P_n)$ coincides with the monoid 
$$
\DP_n=\{\alpha\in\I(\{1,2,\ldots,n\}) \mid |i\alpha-j\alpha|=|i-j|, \mbox{for all $i,j\in\dom(\alpha)$}\}
$$
of all partial isometries on $\{1,2,\ldots,n\}$. 

The study of partial isometries on $\{1,2,\ldots,n\}$ was initiated 
by Al-Kharousi et al.~\cite{AlKharousi&Kehinde&Umar:2014,AlKharousi&Kehinde&Umar:2016}. 
The first of these two papers is dedicated to investigating some combinatorial properties of 
the monoid $\DP_n$ and of its submonoid $\ODP_n$ of all order-preserving (considering the usual order of $\N$) partial isometries, in particular, their cardinalities. The second paper presents the study of some of their algebraic properties, namely Green's structure and ranks. Presentations for both the monoids $\DP_n$ and $\ODP_n$ were given by the first author and Quinteiro in \cite{Fernandes&Quinteiro:2016}. 

The monoid $\DPS_n$ of all partial isometries of a star graph with $n$ vertices ($n\geqslant1$) was considered by the authors in 
\cite{Fernandes&Paulista:2022arxiv}. They determined the rank and size of $\DPS_n$ as well as described its Green's relations. A presentation for $\DPS_n$ was also exhibited in \cite{Fernandes&Paulista:2022arxiv}. 

\smallskip 

Now, for $n\geqslant3$, consider the \textit{cycle graph} 
$$
C_n=(\{1,2,\ldots, n\}, \{\{i,i+1\}\mid i=1,2,\ldots,n-1\}\cup\{\{1,n\}\}) 
$$
with $n$ vertices.  Notice that, cycle graphs and cycle subgraphs play a fundamental role in Graph Theory. 

\smallskip 

This paper is devoted to studying the monoid $\mathcal{DP}(C_n)$ of all partial isometries of $C_n$, 
which from now on we denote simply by $\DPC_n$. Observe that $\DPC_n$ is an inverse submonoid of the symmetric inverse monoid $\I_n$. 

\smallskip 

In Section \ref{basics} we start by giving a key characterization of $\DPC_n$, which allows for significantly simpler proofs of various results presented later. 
Also in this section, a description of the Green's relations of $\DPC_n$ is given and the rank and the cardinal of $\DPC_n$ are calculated. 
Finally, in Section \ref{presenta},  
we determine a presentation for the monoid $\DPC_n$ on $n+2$ generators,
from which we deduce another presentation for $\DPC_n$ on $3$ generators. 

\smallskip 

For general background on Semigroup Theory and standard notations, we refer to Howie's book \cite{Howie:1995}.

\smallskip 

We would like to point out that we made use of computational tools, namely GAP \cite{GAP4}.

\section{Some properties of $\DPC_n$} \label{basics} 

We begin this section by introducing some concepts and notation. 

For $n\in\N$, let $\Omega_n$ be a set with $n$ elements. 
As usual, we denote $\PT(\Omega_n)$, 
$\I(\Omega_n)$ and $\Sym(\Omega_n)$ 
simply by $\PT_n$, 
$\I_n$ and $\Sym_n$, respectively. 

Let $\alpha\in\PT_n$. Recall that the \textit{rank} of $\alpha$ is the size of $\im(\alpha)$.  

Next, suppose that $\Omega_n$ is a chain, e.g. $\Omega_n=\{1<2<\cdots<n\}$. 

A partial transformation $\alpha\in\PT_n$ is called \textit{order-preserving}  
[\textit{order-reversing}] if $x\leqslant y$ implies $x\alpha\leqslant y\alpha$
[$x\alpha\geqslant y\alpha$], for all $x,y \in \dom(\alpha)$.  
It is clear that the product of two order-preserving or of two order-reversing transformations is order-preserving and 
the product of an order-preserving transformation by an
order-reversing transformation, or vice-versa, is 
order-reversing. 
We denote by $\POD_n$ the submonoid of $\PT_n$
whose elements are all order-preserving or order-reversing transformations. 

Let $s=(a_1,a_2,\ldots,a_t)$
be a sequence of $t$ ($t\geqslant0$) elements
from the chain $\Omega_n$. 
We say that $s$ is \textit{cyclic} 
[\textit{anti-cyclic}] if there
exists no more than one index $i\in\{1,\ldots,t\}$ such that
$a_i>a_{i+1}$ [$a_i<a_{i+1}$],
where $a_{t+1}$ denotes $a_1$.
Notice that, the sequence $s$ is cyclic
[anti-cyclic] if and only if $s$ is empty or there exists
$i\in\{0,1,\ldots,t-1\}$ such that 
$a_{i+1}\leqslant a_{i+2}\leqslant \cdots\leqslant a_t\leqslant a_1\leqslant \cdots\leqslant a_i $ 
[$a_{i+1}\geqslant a_{i+2}\geqslant \cdots\geqslant a_t\geqslant a_1\geqslant \cdots\geqslant a_i $] (the index
$i\in\{0,1,\ldots,t-1\}$ is unique unless $s$ is constant and
$t\geqslant2$). We also say that $s$ is \textit{oriented} if $s$ is cyclic or $s$ is anti-cyclic. 
See \cite{Catarino&Higgins:1999,Higgins&Vernitski:2022,McAlister:1998}. 
Given a partial transformation $\alpha\in\PT_n$ such that
$\dom(\alpha)=\{a_1<\cdots<a_t\}$, with $t\geqslant0$, we 
say that $\alpha$ is \textit{orientation-preserving} 
[\textit{orientation-reversing}, \textit{oriented}] if the sequence of its images
$(a_1\alpha,\ldots,a_t\alpha)$ is cyclic [anti-cyclic, oriented].  
It is easy to show 
that the product of two orientation-preserving or of two
orientation-reversing transformations is orientation-preserving and
the product of an orientation-preserving transformation by an
orientation-reversing transformation, or vice-versa, is orientation-reversing. 
We denote by $\POR_n$ the submonoid of $\PT_n$ of all oriented  transformations. 

Notice that $\POD_n\cap\I_n$ and $\POR_n\cap\I_n$ are inverse submonoids of $\I_n$. 

\smallskip 

Let us consider the following permutations of $\Omega_n$ of order $n$ and $2$, respectively: 
$$
g=\begin{pmatrix} 
1&2&\cdots&n-1&n\\
2&3&\cdots&n&1
\end{pmatrix} 
\quad\text{and}\quad  
h=\begin{pmatrix} 
1&2&\cdots&n-1&n\\
n&n-1&\cdots&2&1
\end{pmatrix}. 
$$

It is clear that $g,h\in\POR_n\cap\I_n$. 
Moreover, for $n\geqslant3$, $g$ together with $h$ generate the well-known \textit{dihedral group} $\D_{2n}$ of order $2n$ 
(considered as a subgroup of $\Sym_n$). In fact, for $n\geqslant3$, 
$$
\D_{2n}=\langle g,h\mid g^n=1,h^2=1, hg=g^{n-1}h\rangle=\{1,g,g^2,\ldots,g^{n-1}, h,hg,hg^2,\ldots,hg^{n-1}\} 
$$
and we have 
$$
g^k=\begin{pmatrix} 
1&2&\cdots&n-k&n-k+1&\cdots&n\\
1+k&2+k&\cdots&n&1&\cdots&k
\end{pmatrix}, 
\quad\text{i.e.}\quad  
ig^k=\left\{\begin{array}{lc}
i+k & 1\leqslant i\leqslant n-k\\
i+k-n & n-k+1\leqslant i\leqslant n , 
\end{array}\right.
$$
and 
$$
hg^k=\begin{pmatrix} 
1&\cdots&k&k+1&\cdots&n\\
k&\cdots&1&n&\cdots&k+1
\end{pmatrix}, 
\quad\text{i.e.}\quad  
ihg^k=\left\{\begin{array}{lc}
k-i+1 & 1\leqslant i\leqslant k\\
n+k-i+1 & k+1\leqslant i\leqslant n , 
\end{array}\right.
$$
for $0\leqslant k\leqslant n-1$. 
Observe that, for $n\in\{1,2\}$, the dihedral group $\D_{2n}=\langle g,h\mid g^n=1, h^2=1, hg=g^{n-1}h\rangle$ of order $2n$ 
(also known as the \textit{Klein four-group} for $n=2$) cannot be considered as a subgroup of $\Sym_n$. 
Denote also by $\C_n$ the \textit{cyclic group} of order $n$ generated by $g$, i.e. 
$\C_n=\{1,g,g^2,\ldots,g^{n-1}\}$.  

\medskip 

Until the end of this paper, we will consider $n\geqslant3$. 

\smallskip 

Now, notice that, clearly, we have 
$$
\d_{C_n}(x,y)=\min \{|x-y|,n-|x-y|\} 
= \left\{ \begin{array}{ll}
 |x-y| &\mbox{if  $|x-y|\leqslant\frac{n}{2}$}\\
n-|x-y| &\mbox{if $|x-y|>\frac{n}{2}$} 
\end{array} \right.
$$ 
and so $0\leqslant\d_{C_n}(x,y)\leqslant\frac{n}{2}$,
for all $x,y \in \{1,2,\ldots,n\}$. 

From now on, for any two vertices $x$ and $y$ of $C_n$, we denote the distance $\d_{C_n}(x,y)$ simply by  $\d(x,y)$. 

Let $x,y \in \{1,2,\ldots,n\}$. Observe that 
$$
\d(x,y)=\frac{n}{2}
\quad\Leftrightarrow\quad
|x-y|=\frac{n}{2}
\quad\Leftrightarrow\quad
n-|x-y|=\displaystyle\frac{n}{2}
\quad\Leftrightarrow\quad
|x-y|=n-|x-y|, 
$$
in which case $n$ is even, and 
\begin{equation}\label{d1}
|\left\{z\in \{1,2,\ldots,n\}\mid \d(x,z)=d\right\}|=
 \left\{ \begin{array}{ll}
1 &\mbox{if  $d=\frac{n}{2}$}\\
2 &\mbox{if $d<\frac{n}{2}$,} 
\end{array} \right.
\end{equation}
for all $1\leqslant d \leqslant\frac{n}{2}$. Moreover, it is a routine matter to show that 
$$
D=\left\{z\in \{1,2,\ldots,n\}\mid \d(x,z)=d\right\}=\left\{z\in \{1,2,\ldots,n\}\mid \d(y,z)=d'\right\}
$$
implies 
\begin{equation}\label{d2}
\d(x,y)=\left\{ \begin{array}{ll}
\mbox{$0$ (i.e. $x=y$)} &\mbox{if  $|D|=1$}\\
\frac{n}{2} &\mbox{if $|D|=2$,} 
\end{array} \right.
\end{equation}
for all $1\leqslant d,d' \leqslant\frac{n}{2}$. 

\medskip 

Recall that $\DP_n$ is an inverse submonoid of $\POD_n\cap\I_n$. This is an easy fact to prove and was observed by Al-Kharousi et al. in  \cite{AlKharousi&Kehinde&Umar:2014,AlKharousi&Kehinde&Umar:2016}. A similar result is also valid for $\DPC_n$ and $\POR_n\cap\I_n$,
as we will deduce below. 

First, notice that, it is easy to show that both permutations $g$ and $h$ of $\Omega_n$ belong to $\DPC_n$ 
and so the dihedral group $\D_{2n}$ is contained in $\DPC_n$. 
Furthermore, as we prove next, the elements of $\DPC_n$ are precisely the restrictions of the permutations of the dihedral group $\D_{2n}$.  
This is a key characterization of $\DPC_n$ that will allow us to prove in a simpler way some of the results that we present later in this paper. 
Observe that 
$$
\alpha=\sigma|_{\dom(\alpha)} 
\quad\Leftrightarrow\quad
\alpha=\id_{\dom(\alpha)} \sigma
\quad\Leftrightarrow\quad
\alpha=\sigma\id_{\im(\alpha)},  
$$
for all $\alpha\in\PT_n$ and $\sigma\in\I_n$. 

\begin{lemma}\label{fundlemma} 
Let $\alpha \in \PT_n$. Then $\alpha \in\DPC_n$ if and only if there exists $\sigma \in \D_{2n}$ 
such that $\alpha=\sigma|_{\dom(\alpha)}$. 
Furthermore, for $\alpha \in \DPC_n$, one has: 
\begin{enumerate} 
\item If either $|\dom(\alpha)|= 1$ or $|\dom(\alpha)|= 2$ and $\d(\min \dom(\alpha),\max \dom(\alpha))=\frac{n}{2}$ 
(in which case $n$ is even), 
then there exist exactly two (distinct) permutations $\sigma,\sigma' \in\D_{2n}$ such that $\alpha= \sigma|_{\dom(\alpha)} = \sigma'|_{\dom(\alpha)}$;

\item If either $|\dom(\alpha)|= 2$ and $\d(\min \dom(\alpha),\max \dom(\alpha)) \neq \frac{n}{2}$ or $|\dom(\alpha)|\geqslant 3$, 
then there exists exactly one permutation $\sigma \in\mathcal{D}_{2n}$ such that $\alpha= \sigma|_{\dom(\alpha)}$.
\end{enumerate}
\end{lemma}
\begin{proof}
Let $\alpha \in \PT_n$. 

\smallskip 

If $\alpha=\sigma|_{\dom(\alpha)}$, for some $\sigma \in \D_{2n}$, then $\alpha\in\DPC_n$, since $\D_{2n}\subseteq\DPC_n$ 
and, clearly, any restriction of an element of $\DPC_n$ also belongs to $\DPC_n$. 

\smallskip

Conversely, let us suppose that $\alpha\in\DPC_n$. 

First, observe that, for each pair $1\leqslant i,j\leqslant n$, there exists a unique $k\in\{0,1,\ldots,n-1\}$ such that $ig^k=j$ and 
there exists a unique $\ell\in\{0,1,\ldots,n-1\}$ such that $ihg^\ell=j$. In fact, for $1\leqslant i,j\leqslant n$ and $k,\ell\in\{0,1,\ldots,n-1\}$, 
it is easy to show that: 
\begin{description}
\item if $i\leqslant j$ then $ig^k=j$ if and only if $k=j-i$; 

\item if $i>j$ then $ig^k=j$ if and only if $k=n+j-i$; 

\item if $i+j\leqslant n$ then $ihg^\ell=j$ if and only if $\ell=i+j-1$; 

\item if $i+j > n$ then $ihg^\ell=j$ if and only if $\ell=i+j-1-n$. 
\end{description}
Therefore, we may conclude immediately that:
\begin{enumerate}
\item any nonempty transformation of $\DPC_n$ has at most two extensions in $\D_{2n}$ and, if there are two distinct,
one must be an orientation-preserving transformation and the other an orientation-reversing transformation;

\item any transformation of $\DPC_n$ with rank $1$ has two distinct extensions in $\D_{2n}$ 
(one being an orientation-preserving transformation and the other an orientation-reversing transformation). 
\end{enumerate} 

Notice that, as $g^n=g^{-n}=1$, we also have $ig^{j-i}=j$ and $ihg^{i+j-1}=j$, for all $1\leqslant i,j\leqslant n$. 

\smallskip 

Next, suppose that $\dom(\alpha)=\{i_1,i_2\}$. 
Then, there exist $\sigma\in\C_n$ and $\xi\in\D_{2n}\setminus\C_n$ (both unique) such that 
$i_1\sigma=i_1\alpha=i_1\xi$. 
Take 
$D=\left\{z\in \{1,2,\ldots,n\}\mid \d(i_1\alpha,z)=\d(i_1,i_2)\right\}$. 
Then $1\leqslant |D|\leqslant 2$ and $i_2\alpha,i_2\sigma,i_2\xi\in D$. 

Suppose that $i_2\sigma=i_2\xi$ and let $j_1=i_1\sigma$ and $j_2=i_2\sigma$. Then 
$\sigma=g^{j_1-i_1}=g^{j_2-i_2}$ and $\xi=hg^{i_1+j_1-1}=hg^{i_2+j_2-1}$. 
Hence, we have $j_1-i_1=j_2-i_2$ or $j_1-i_1=j_2-i_2\pm n$, from the first equality, 
and $i_1+j_1=i_2+j_2$ or $i_1+j_1=i_2+j_2\pm n$, from the second. 
Since $i_1\neq i_2$ and $i_2-i_1\neq n$, 
it a routine matter to conclude that the only possibility is to have $i_2-i_1=\frac{n}{2}$ (in which case $n$ is even). Thus $\d(i_1,i_2)=\frac{n}{2}$. 
By (\ref{d1}) it follows that $|D|=1$ and so $i_2\alpha=i_2\sigma=i_2\xi$, i.e. $\alpha$ is extended by both $\sigma$ and $\xi$. 

If $i_2\sigma\neq i_2\xi$ then $|D|=2$ (whence $\d(i_1,i_2)<\frac{n}{2}$) and so either $i_2\alpha=i_2\sigma$ or $i_2\alpha=i_2\xi$. 
In this case, $\alpha$ is extended by exactly one permutation of $\D_{2n}$. 

\smallskip 

Now, suppose that $\dom(\alpha)=\{i_1<i_2<\cdots <i_k\}$, for some $3\leqslant k\leqslant n-1$. 
Since $\sum_{p=1}^{k-1}(i_{p+1}-i_p) = i_k-i_1<n$, then 
there exists at most one index $1\leqslant p\leqslant k-1$ such that $i_{p+1}-i_p\geqslant\frac{n}{2}$. 
Therefore, we may take $i,j\in\dom(\alpha)$ such that $i\neq j$ and $\d(i,j)\neq\frac{n}{2}$ and so, 
as $\alpha|_{\{i,j\}}\in\DPC_n$, by the above deductions, there exists a unique $\sigma\in\D_{2n}$ such that 
$\sigma|_{\{i,j\}}=\alpha|_{\{i,j\}}$. 
Let $\ell\in\dom(\alpha)\setminus\{i,j\}$. 
Then 
$$
\ell\alpha,\ell\sigma\in \left\{z\in \{1,2,\ldots,n\}\mid \d(i\alpha,z)=\d(i,\ell)\right\}\cap\left\{z\in \{1,2,\ldots,n\}\mid \d(j\alpha,z)=\d(j,\ell)\right\}.
$$ 
In order to obtain a contradiction, suppose that $\ell\alpha\neq\ell\sigma$. 
Therefore, by (\ref{d1}), we have  
$$
\left\{z\in \{1,2,\ldots,n\}\mid \d(i\alpha,z)=\d(i,\ell)\right\} = 
\left\{\ell\alpha,\ell\sigma\right\}= 
\left\{z\in \{1,2,\ldots,n\}\mid \d(j\alpha,z)=\d(j,\ell)\right\}
$$
and so, by (\ref{d2}), $\d(i,j)=\d(i\alpha,j\alpha)=\frac{n}{2}$, which is a contradiction. 
Hence $\ell\alpha=\ell\sigma$. Thus $\sigma$ is the unique permutation of $\D_{2n}$ such that 
$\alpha= \sigma|_{\dom(\alpha)}$, as required.
\end{proof} 

Bearing in mind the previous lemma, it seems appropriate to designate $\DPC_n$ by \textit{dihedral inverse monoid} on $\Omega_n$.

\smallskip 

Since $\D_{2n}\subseteq\POR_n\cap\I_n$, which contains all the restrictions of its elements, we have immediately: 

\begin{corollary}\label{dpcpopi} 
The monoid $\DPC_n$ is contained in $\POR_n\cap\I_n$. 
\end{corollary}

Observe that, as $\D_{2n}$ is the group of units of $\POR_n\cap\I_n$ (see \cite{Fernandes&Gomes&Jesus:2004,Fernandes&Gomes&Jesus:2009}), 
then $\D_{2n}$ also has to be the group of units of $\DPC_n$. 

\medskip 

Next, recall that, given an inverse submonoid $M$ of $\I_n$, it is well known that 
the Green's relations $\mathscr{L}$, $\mathscr{R}$ and $\mathscr{H}$
of $M$ can be described as following: for $\alpha, \beta \in M$,
\begin{itemize}
\item $\alpha \mathscr{L} \beta$ if and only if $\im(\alpha) = \im(\beta)$;

\item $\alpha \mathscr{R} \beta$ if and only if $\dom(\alpha) = \dom(\beta)$;

\item $\alpha \mathscr{H} \beta $ if and only if $\im(\alpha) = \im(\beta)$ and $\dom(\alpha) = \dom(\beta)$.
\end{itemize}
In $\I_n$ we also have 
\begin{itemize}
\item $\alpha \mathscr{J} \beta$ if and only if $|\dom(\alpha)| = |\dom(\beta)|$ (if and only if $|\im(\alpha)| = |\im(\beta)|$). 
\end{itemize}

Since $\DPC_n$ is an inverse submonoid of $\I_n$, 
it remains to describe its Green's relation $\mathscr{J}$. 
In fact, it is a routine matter to show that: 

\begin{proposition} \label{greenJ}
Let $\alpha, \beta \in \DPC_n$. Then $\alpha \mathscr{J} \beta$ if and only if one of the following properties is satisfied:
\begin{enumerate}
\item $|\dom(\alpha)|=|\dom(\beta)|\leqslant1$;
\item $|\dom(\alpha)|=|\dom(\beta)|=2$ and $\d(i_1,i_2)=\d(i'_1,i'_2)$, 
where $\dom(\alpha)=\{i_1,i_2\}$ and $\dom(\beta)=\{i'_1,i'_2\}$; 
\item $|\dom(\alpha)|=|\dom(\beta)|=k\geqslant3$ and there exists $\sigma\in\D_{2k}$ such that 
$$
\begin{pmatrix} 
i'_1&i'_2&\cdots&i'_k\\
i_{1\sigma}&i_{2\sigma}&\cdots&i_{k\sigma}
\end{pmatrix} \in\DPC_n, 
$$
where $\dom(\alpha)=\{i_1<i_2<\dots<i_k\}$ and $\dom(\beta)=\{i'_1<i'_2<\cdots<i'_k\}$.  
\end{enumerate}
\end{proposition}

An alternative description of $\mathscr{J}$ can be found in second author's M.Sc.~thesis \cite{Paulista:2022}. 

\medskip 

Next, we count the number of elements of $\DPC_n$. 

\begin{theorem}
One has $|\DPC_n| = n2^{n+1}-\frac{(-1)^n+5}{4}n^2-2n+1$. 
\end{theorem}
\begin{proof}
Let $\A_i=\{\alpha\in\DPC_n\mid |\dom(\alpha)|=i\}$, for $i=0,1,\ldots,n$.  
Since the sets $\A_0,\A_1,\ldots,\A_n$ are pairwise disjoints, 
we get $|\DPC_n|=\sum_{i=0}^{n} |\A_i|$. 

Clearly, $\A_0=\{\emptyset\}$ and $\A_1=\{\binom{i}{j}\mid  1\leqslant i,j\leqslant n\}$, 
whence $|\A_0|=1$ and $|\A_1|=n^2$. Moreover, for $i\geqslant3$, by Lemma \ref{fundlemma}, 
we have as many elements in $\A_i$ as there are restrictions of rank $i$ of permutations of $\D_{2n}$, i.e. we have 
$\binom{n}{i}$ distinct elements of $\A_i$ for each permutation of $\D_{2n}$, whence $|\A_i|=2n\binom{n}{i}$. 
Similarly, for an odd $n$, by Lemma \ref{fundlemma}, we have $|\A_2|=2n\binom{n}{2}$. 
On the other hand, if $n$ is even, also by Lemma \ref{fundlemma}, 
we have as many elements in $\A_2$ as there are restrictions of rank $2$ of permutations of $\D_{2n}$ 
minus the number of elements of $\A_2$ that have two distinct extensions in $\D_{2n}$, i.e. 
$|\A_2|=2n\binom{n}{2}-|\B_2|$, where 
$$
\B_2=\{\alpha\in\DPC_n\mid |\mbox{$\dom(\alpha)|=2$ and $\d(\min \dom(\alpha),\max \dom(\alpha))=\frac{n}{2}$}\}. 
$$
It is easy to check that 
$$
\B_2=\left\{
\begin{pmatrix} 
i&i+\frac{n}{2}\\
j&j+\frac{n}{2}
\end{pmatrix},
\begin{pmatrix} 
i&i+\frac{n}{2}\\
j+\frac{n}{2}&j
\end{pmatrix}
\mid 
1\leqslant i,j\leqslant \frac{n}{2}
\right\},
$$
whence $|\B_2|=2(\frac{n}{2})^2=\frac{1}{2}n^2$. 
Therefore
$$
|\DPC_n|= 
\left\{\begin{array}{ll} 
1+n^2+2n\sum_{i=2}^{n}\binom{n}{i} & \mbox{if $n$ is odd}
\\\\
1+n^2+2n\sum_{i=2}^{n}\binom{n}{i} -\frac{1}{2}n^2 & \mbox{if $n$ is even}
\end{array}\right. 
= 
\left\{\begin{array}{ll} 
n2^{n+1}-n^2-2n+1 & \mbox{if $n$ is odd}
\\\\
n2^{n+1}-\frac{3}{2}n^2-2n+1 & \mbox{if $n$ is even}, 
\end{array}\right. 
$$
as required. 
\end{proof}

\medskip

We finish this section by deducing that $\DPC_n$ has rank $3$.

Let 
$$
e_i=\id_{\Omega_n\setminus\{i\}}=
\begin{pmatrix} 1&\cdots&i-1&i+1&\cdots&n\\
1&\cdots&i-1&i+1&\cdots&n
\end{pmatrix}\in\DPC_n, 
$$
for $i=1,2,\ldots,n$. Clearly, for $1\leqslant i,j\leqslant n$, we have $e_i^2=e_i$ and $e_ie_j=\id_{\Omega_n\setminus\{i,j\}}=e_je_i$. 
More generally, for any $X\subseteq\Omega_n$, we get $\Pi_{i\in X}e_i=\id_{\Omega_n\setminus X}$. 

Now, take $\alpha\in\DPC_n$. Then, by Lemma \ref{fundlemma}, $\alpha=h^ig^j|_{\dom(\alpha)}$, 
for some $i\in\{0,1\}$ and $j\in\{0,1,\ldots,n-1\}$. 
Hence $\alpha=h^ig^j\id_{\im(\alpha)}=h^ig^j\Pi_{k\in\Omega_n\setminus\im(\alpha)}e_k$. 
Therefore
$
\{g,h,e_1,e_2,\ldots,e_n\}
$
is a generating set of $\DPC_n$. Since 
$e_j=hg^{j-1}e_nhg^{j-1}$ for all $j\in\{1,2,\ldots,n\}$, 
it follows that $\{g,h,e_n\}$ is also a generating set of $\DPC_n$. 
As $\D_{2n}$ is the group of units of $\DPC_n$, which is a group with rank $2$, 
the monoid $\DPC_n$ cannot be generated by less than three elements. So, we have: 

\begin{theorem}
The rank of the monoid $\DPC_n$ is $3$.
\end{theorem}

\section{Presentations for $\DPC_n$}\label{presenta}

In this section, we aim to determine a presentation for $\DPC_n$. 
In fact, we first determine a presentation of $\DPC_n$ on $n+2$ generators and then, 
by applying applying \textit{Tietze transformations}, we deduce a presentation for $\DPC_n$ on $3$ generators. 

\smallskip 

We begin this section by recalling some notions related to the concept of a monoid presentation.

\smallskip 

Let $A$ be an alphabet and consider the free monoid $A^*$ generated by $A$. 
The elements of $A$ and of $A^*$ are called \textit{letters} and \textit{words}, respectively. 
The empty word is denoted by $1$ and we write $A^+$ to express $A^*\setminus\{1\}$.  
A pair $(u,v)$ of $A^*\times A^*$ is called a
\textit{relation} of $A^*$ and it is usually represented by $u=v$. 
To avoid confusion, given $u, v\in A^*$, we will write $u\equiv v$, instead
of  $u=v$, whenever we want to state precisely that $u$ and $v$
are identical words of $A^*$.  
A relation $u=v$ of $A^*$ is said to be a \textit{consequence} of $R$ if $u{\hspace{.11em}}\rho_R{\hspace{.11em}}v$. 

Let $X$ be a generating set of $M$ and let $\phi: A\longrightarrow M$ be an injective mapping 
such that $A\phi=X$. 
Let $\varphi: A^*\longrightarrow M$ be the (surjective) homomorphism of monoids that extends $\phi$ to $A^*$. 
We say that $X$ satisfies (via $\varphi$) a relation $u=v$ of $A^*$ if $u\varphi=v\varphi$. 
For more details see
\cite{Lallement:1979} or \cite{Ruskuc:1995}. 

A direct method to find a presentation for a monoid
is described by the following well-known result (e.g.  see \cite[Proposition 1.2.3]{Ruskuc:1995}).  

\begin{proposition}\label{provingpresentation} 
Let $M$ be a monoid generated by a set $X$, let $A$ be an alphabet 
and let $\phi: A\longrightarrow M$ be an injective mapping 
such that $A\phi=X$. 
Let $\varphi:A^*\longrightarrow M$ be the (surjective) homomorphism 
that extends $\phi$ to $A^*$ and let $R\subseteq A^*\times A^*$.
Then $\langle A\mid R\rangle$ is a presentation for $M$ if and only
if the following two conditions are satisfied:
\begin{enumerate}
\item
The generating set $X$ of $M$ satisfies (via $\varphi$) all the relations from $R$;  
\item 
If $u,v\in A^*$ are any two words such that 
the generating set $X$ of $M$ satisfies (via $\varphi$) the relation $u=v$ then $u=v$ is a consequence of $R$.
\end{enumerate} 
\end{proposition}

\smallskip

Given a presentation for a monoid, another method to find a new
presentation consists in applying Tietze transformations. For a
monoid presentation $\langle A\mid R\rangle$, the 
four \emph{elementary Tietze transformations} are:

\begin{description}
\item(T1)
Adding a new relation $u=v$ to $\langle A\mid R\rangle$,
provided that $u=v$ is a consequence of $R$;
\item(T2)
Deleting a relation $u=v$ from $\langle A\mid R\rangle$,
provided that $u=v$ is a consequence of $R\backslash\{u=v\}$;
\item(T3)
Adding a new generating symbol $b$ and a new relation $b=w$, where
$w\in A^*$;
\item(T4)
If $\langle A\mid R\rangle$ possesses a relation of the form
$b=w$, where $b\in A$, and $w\in(A\backslash\{b\})^*$, then
deleting $b$ from the list of generating symbols, deleting the
relation $b=w$, and replacing all remaining appearances of $b$ by
$w$.
\end{description}

The next result is well-known (e.g. see \cite{Ruskuc:1995}): 

\begin{proposition} \label{tietze}
Two finite presentations define the same monoid if and only if one
can be obtained from the other by a finite number of elementary
Tietze transformations $(T1)$, $(T2)$, $(T3)$ and $(T4)$.  
\end{proposition}

\medskip 

Now, consider the alphabet $A=\{g,h,e_1,e_2,\ldots,e_n\}$ and the following set $R$ formed by the following monoid relations: 
\begin{description}
\item $(R_1)$ $g^n=1$, $h^2=1$ and $hg=g^{n-1}h$; 

\item $(R_2)$ $e_i^2=e_i$, for $1\leqslant i\leqslant n$;

\item $(R_3)$ $e_ie_j=e_je_i$, for $1\leqslant i<j\leqslant n$; 

\item $(R_4)$ $ge_1=e_ng$ and $ge_{i+1}=e_ig$, for $1\leqslant i\leqslant n-1$; 

\item $(R_5)$ $he_i=e_{n-i+1}h$, for $1\leqslant i\leqslant n$;

\item $(R_6^\text{o})$ $hge_2e_3\cdots e_n=e_2e_3\cdots e_n$, if $n$ is odd; 

\item $(R_6^\text{e})$ $hge_2\cdots e_{\frac{n}{2}}e_{\frac{n}{2}+2}\cdots e_n=e_2\cdots e_{\frac{n}{2}}e_{\frac{n}{2}+2}\cdots e_n$ 
and $he_1e_2\cdots e_n=e_1e_2\cdots e_n$, if $n$ is even.
\end{description}
Observe that $|R|=\frac{n^2+5n+9+(-1)^n}{2}$. 

\smallskip 

We aim to show that the monoid $\DPC_n$ is defined by the presentation $\langle A \mid R\rangle$. 

\smallskip 

Let $\phi:A\longrightarrow \DPC_n$ be the mapping defined by
$$
g\phi=g ,\quad h\phi=h ,\quad e_i\phi=e_i, \mbox{~for $1\leqslant i\leqslant n$}, 
$$
and let $\varphi:A^*\longrightarrow \DPC_n$ be the homomorphism of monoids that extends $\phi$ to $A^*$. 
Notice that we are using the same symbols for the letters of the alphabet $A$ and for the generating set of $\DPC_n$,
which simplifies notation and, within the context, will not cause ambiguity. 

\smallskip 

It is a routine matter to check the following lemma. 

\begin{lemma}\label{genrel}
The set of generators $\{g,h,e_1,e_2,\ldots,e_n\}$ of $\DPC_n$ satisfies (via $\varphi$) all the relations from $R$.
\end{lemma}

Observe that this result assures us that, if $u,v\in A^*$ are such that the relation $u=v$ is a consequence of $R$, 
then $u\varphi=v\varphi$. 

\smallskip 

Next, in order to prove that any relation satisfied by the generating set of $\DPC_n$ is a consequence of $R$, 
we first present a series of three lemmas. 

In what follows, we denote the congruence $\rho_R$ of $A^*$ simply by $\rho$. 

\begin{lemma}\label{pre0}
If $n$ is even then the relation 
$$
hg^{2j-1}e_1\cdots e_{j-1}e_{j+1}\cdots e_{j+\frac{n}{2}-1}e_{j+\frac{n}{2}+1}\cdots e_n =
e_1\cdots e_{j-1}e_{j+1}\cdots e_{j+\frac{n}{2}-1}e_{j+\frac{n}{2}+1}\cdots e_n
$$  
is a consequence of $R$, for $1\leqslant j\leqslant \frac{n}{2}$. 
\end{lemma}
\begin{proof} 
We proceed by induction on $j$.

Let $j=1$. 
Then $hge_2\cdots e_{\frac{n}{2}}e_{\frac{n}{2}+2}\cdots e_n=e_2\cdots e_{\frac{n}{2}}e_{\frac{n}{2}+2}\cdots e_n$ 
is a relation of $R$. 

Next, suppose that $hg^{2j-1}e_1\cdots e_{j-1}e_{j+1}\cdots e_{j+\frac{n}{2}-1}e_{j+\frac{n}{2}+1}\cdots e_n =
e_1\cdots e_{j-1}e_{j+1}\cdots e_{j+\frac{n}{2}-1}e_{j+\frac{n}{2}+1}\cdots e_n$,  
for some $1\leqslant j\leqslant \frac{n}{2}-1$. 
Then
$$
\begin{array}{cll}
& hg^{2(j+1)-1}e_1\cdots e_je_{j+2}\cdots e_{j+\frac{n}{2}}e_{j+\frac{n}{2}+2}\cdots e_n & \\
  \equiv & hg^{2j+1}e_1\cdots e_je_{j+2}\cdots e_{j+\frac{n}{2}}e_{j+\frac{n}{2}+2}\cdots e_n & \\
  \rho & hg^{2j}e_nge_2\cdots e_je_{j+2}\cdots e_{j+\frac{n}{2}}e_{j+\frac{n}{2}+2}\cdots e_n & \mbox{(by $R_4$)}\\ 
  \rho & hgg^{2j-1}e_n e_1\cdots e_{j-1}e_{j+1}\cdots e_{j+\frac{n}{2}-1}e_{j+\frac{n}{2}+1}\cdots e_{n-1}g & \mbox{(by $R_4$)}\\ 
  \rho & g^{n-1}hg^{2j-1}e_1\cdots e_{j-1}e_{j+1}\cdots e_{j+\frac{n}{2}-1}e_{j+\frac{n}{2}+1}\cdots e_ng & \mbox{(by $R_1$ and $R_3$)}\\ 
  \rho & g^{n-1}e_1\cdots e_{j-1}e_{j+1}\cdots e_{j+\frac{n}{2}-1}e_{j+\frac{n}{2}+1}\cdots e_ng& \mbox{(by the induction hyphotesis)}\\ 
  \rho & g^{n-1}e_1\cdots e_{j-1}e_{j+1}\cdots e_{j+\frac{n}{2}-1}e_{j+\frac{n}{2}+1}\cdots e_{n-1}ge_1 & \mbox{(by $R_4$)}\\ 
  \rho &g^{n-1}ge_2\cdots e_je_{j+2}\cdots e_{j+\frac{n}{2}}e_{j+\frac{n}{2}+2}\cdots e_ne_1 & \mbox{(by $R_4$)}\\ 
  \rho &e_1\cdots e_je_{j+2}\cdots e_{j+\frac{n}{2}}e_{j+\frac{n}{2}+2}\cdots e_n & \mbox{(by $R_1$ and $R_3$)}, 
\end{array}
$$
as required. 
\end{proof} 

\begin{lemma}\label{pre1}
The relation $hg^{2i-1}e_1\cdots e_{i-1}e_{i+1}\cdots e_n=e_1\cdots e_{i-1}e_{i+1}\cdots e_n$ 
is a consequence of $R$, for $1\leqslant i\leqslant n$. 
\end{lemma}
\begin{proof} 
We proceed by induction on $i$.

Let $i=1$. If $n$ is odd then $hge_2e_3\cdots e_n=e_2e_3\cdots e_n$ is a relation of $R$. So, suppose that $n$ is even. 
Then $hge_2\cdots e_{\frac{n}{2}}e_{\frac{n}{2}+2}\cdots e_n=e_2\cdots e_{\frac{n}{2}}e_{\frac{n}{2}+2}\cdots e_n$ is a relation of $R$,
whence 
$$
hge_2\cdots e_{\frac{n}{2}}e_{\frac{n}{2}+2}\cdots e_ne_{\frac{n}{2}+1}\ro e_2\cdots e_{\frac{n}{2}}e_{\frac{n}{2}+2}\cdots e_ne_{\frac{n}{2}+1}
$$ 
and so $hge_2e_3\cdots e_n=e_2e_3\cdots e_n$, by $R_3$. 

Now, suppose that $hg^{2i-1}e_1\cdots e_{i-1}e_{i+1}\cdots e_n\ro e_1\cdots e_{i-1}e_{i+1}\cdots e_n$,  
for some $1\leqslant i\leqslant n-1$. 
Then (with steps similar to the previous proof), we have 
$$
\begin{array}{rcll}
hg^{2(i+1)-1}e_1\cdots e_{i}e_{i+2}\cdots e_n & \equiv & hg^{2i+1}e_1\cdots e_{i}e_{i+2}\cdots e_n & \\
 & \rho & hg^{2i}e_nge_2\cdots e_{i}e_{i+2}\cdots e_n & \mbox{(by $R_4$)}\\ 
 & \rho & hgg^{2i-1}e_ne_1\cdots e_{i-1}e_{i+1}\cdots e_{n-1}g & \mbox{(by $R_4$)}\\ 
 & \rho & g^{n-1}hg^{2i-1}e_1\cdots e_{i-1}e_{i+1}\cdots e_ng & \mbox{(by $R_1$ and $R_3$)}\\ 
 & \rho & g^{n-1}e_1\cdots e_{i-1}e_{i+1}\cdots e_n  g & \mbox{(by the induction hyphotesis)}\\ 
 & \rho & g^{n-1}e_1\cdots e_{i-1}e_{i+1}\cdots e_{n-1}ge_1 & \mbox{(by $R_4$)}\\ 
 & \rho &g^{n-1}ge_2\cdots e_{i}e_{i+2}\cdots e_ne_1 & \mbox{(by $R_4$)}\\ 
 & \rho &e_1\cdots e_{i}e_{i+2}\cdots e_n & \mbox{(by $R_1$ and $R_3$)}, 
\end{array}
$$
as required. 
\end{proof} 

\begin{lemma}\label{pre2}
The relation $h^\ell g^me_1e_2\cdots e_n=e_1e_2\cdots e_n$  is a consequence of $R$,  for $\ell,m\geqslant0$. 
\end{lemma}
\begin{proof} 
First, we prove that the relation $he_1e_2\cdots e_n=e_1e_2\cdots e_n$ is a consequence of $R$. 
Since this relation belongs to $R$ for an $even$ $n$, it remains to show that 
$he_1e_2\cdots e_n\ro e_1e_2\cdots e_n$ for $n$ odd. 

Suppose that $n$ is odd. Hence, from $R_6^\text{o}$, 
we have $hge_2e_3\cdots e_ne_1\ro e_2e_3\cdots e_ne_1$, 
so $hge_1e_2\cdots e_n\ro e_1e_2\cdots e_n$ (by $R_3$), whence  
$ge_1e_2\cdots e_n\ro he_1e_2\cdots e_n$ (by $R_1$) and then 
$(ge_1e_2\cdots e_n)^n\ro (he_1e_2\cdots e_n)^n$. 
Now, by applying relations $R_4$ and $R_3$, we have 
$$
ge_1e_2\cdots e_n\ro e_nge_2\cdots e_n\ro e_ne_1\cdots e_{n-1}g \ro e_1e_2\cdots e_ng,
$$
whence $(ge_1e_2\cdots e_n)^n\ro g^n(e_1e_2\cdots e_n)^n\ro e_1e_2\cdots e_n$, by relations $R_1$, $R_3$ and $R_2$. 
On the other hand, 
by applying relations $R_5$ and $R_3$, we get 
$$
he_1e_2\cdots e_n\ro e_ne_{n-1}\cdots e_1h\ro e_1e_2\cdots e_nh,
$$
whence $(he_1e_2\cdots e_n)^n\ro h^n(e_1e_2\cdots e_n)^n\ro h e_1e_2\cdots e_n$, 
by relations $R_1$, $R_3$ and $R_2$, since $n$ is odd. 
Therefore $he_1e_2\cdots e_n\ro e_1e_2\cdots e_n$. 

\smallskip 

Secondly, we prove that the relation $ge_1e_2\cdots e_n=e_1e_2\cdots e_n$ is a consequence of $R$. In fact, we have 
$$
\begin{array}{rcll}
ge_1e_2\cdots e_n & \rho & ge_1hge_2\cdots e_n & \mbox{(by Lemma \ref{pre1})}\\
 & \rho & e_nghge_2\cdots e_n & \mbox{(by $R_4$)}\\
 & \rho & e_ngg^{n-1}he_2\cdots e_n & \mbox{(by $R_1$)}\\
 & \rho & e_nhe_2\cdots e_n & \mbox{(by $R_1$)}\\
 & \rho & he_1e_2\cdots e_n & \mbox{(by $R_5$)}\\ 
 & \rho & e_1e_2\cdots e_n & \mbox{(by the first part).} 
\end{array}
$$

\smallskip 

Now, clearly, for $\ell,m\geqslant0$, $h^\ell g^me_1e_2\cdots e_n\ro e_1e_2\cdots e_n$ follows immediately from 
$ge_1e_2\cdots e_n\ro e_1e_2\cdots e_n$ and $he_1e_2\cdots e_n\rho e_1e_2\cdots e_n$, 
which concludes the proof of the lemma. 
\end{proof} 

We are now in a position to prove the following result.

\begin{theorem}\label{firstpres} 
The monoid $\DPC_n$ is defined by the presentation $\langle A \mid R\rangle$ on $n+2$ generators. 
\end{theorem} 
\begin{proof} In view of Proposition \ref{provingpresentation} and Lemma \ref{genrel}, 
it remains to prove that any relation satisfied by the generating set $\{g,h,e_1,e_2,\ldots,e_n\}$ of $\DPC_n$ is a consequence of $R$. 

Let $u,v\in A^*$ be such that $u\varphi=v\varphi$. We aim to show that $u\ro v$. 

Take $\alpha=u\varphi$. 

It is clear that relations $R_1$ to $R_5$ allow us to deduce that $u\ro h^\ell g^m e_{i_1}\cdots e_{i_k}$, 
for some $\ell\in\{0,1\}$, $m\in\{0,1,\ldots,n-1\}$, $1\leqslant i_1 < \cdots < i_k\leqslant n$ and $0\leqslant k\leqslant n$.  
Similarly, we have $v\ro h^{\ell'} g^{m'} e_{i'_1}\cdots e_{i'_k}$, 
for some $\ell'\in\{0,1\}$, $m'\in\{0,1,\ldots,n-1\}$, $1\leqslant i'_1 < \cdots < i'_{k'}\leqslant n$ and $0\leqslant k'\leqslant n$.  

Since $\alpha=h^\ell g^m e_{i_1}\cdots e_{i_k}$, it follows that $\im(\alpha)=\Omega_n\setminus\{i_1,\ldots,i_k\}$ and 
$\alpha=h^\ell g^m|_{\dom(\alpha)}$. Similarly, as also $\alpha=v\varphi$, 
from $\alpha= h^{\ell'} g^{m'} e_{i'_1}\cdots e_{i'_k}$, 
we get $\im(\alpha)=\Omega_n\setminus\{i'_1,\ldots,i'_{k'}\}$ and $\alpha= h^{\ell'} g^{m'}|_{\dom(\alpha)}$. 
Hence $k'=k$ and $\{i'_1,\ldots,i'_k\}=\{i_1,\ldots,i_k\}$. 
Furthermore, 
if either $k=n-2$ and $\d(\min \dom(\alpha),\max \dom(\alpha)) \neq \frac{n}{2}$ or $k\leqslant n-3$, 
by Lemma \ref{fundlemma}, we obtain $\ell'=\ell$ and $m'=m$ and so $u\ro h^\ell g^m e_{i_1}\cdots e_{i_k} \ro v$. 

\smallskip 

If $h^{\ell'}g^{m'} = h^\ell g^m$ (even as elements of $\D_{2n}$) then $\ell'=\ell$ and $m'=m$ and so   
we get again $u\ro   h^\ell g^m e_{i_1}\cdots e_{i_k} \ro v$. 

Therefore, let us suppose that $h^{\ell'}g^{m'}\neq h^\ell g^m$.  
Hence, by Lemma \ref{fundlemma}, we may conclude that $\alpha=\emptyset$ or $\ell'=\ell-1$ or $\ell'=\ell+1$. 

If $\alpha=\emptyset$, i.e. $k=n$, then 
$u\ro h^\ell g^m e_1e_2\cdots e_n \ro e_1e_2\cdots e_n \ro h^{\ell'}g^{m'} e_1e_2\cdots e_n \ro v$, 
by Lemma \ref{pre2}. 

Thus, we may suppose that $\alpha\neq\emptyset$ and, without loss of generality, also that $\ell'=\ell+1$, i.e. $\ell=0$ and $\ell'=1$. 

\smallskip 

Let $k=n-2$ and admit that $\d(\min \dom(\alpha),\max \dom(\alpha)) = \frac{n}{2}$ (in which case $n$ is even). 

Let 
$\alpha=\begin{pmatrix} 
i_1&i_2\\
j_1&j_2
\end{pmatrix}
$, 
with $1\leqslant i_1<i_2\leqslant n$. Then $i_2-i_1=\frac{n}{2}=\d(i_1,i_2)=\d(j_1,j_2)=|j_2-j_1|$ 
and so $j_2\in\{j_1 -\frac{n}{2}, j_1 +\frac{n}{2}\}$. 
Let $j=\min\{j_1, j_2\}$ (notice that  $1\leqslant j\leqslant\frac{n}{2}$) and $i=j\alpha^{-1}$. 
Hence $\im(\alpha)=\{j,j+\frac{n}{2}\}$ and 
$\alpha=g^{n+j-i}|_{\dom(\alpha)}=hg^{i+j-1-n}|_{\dom(\alpha)}$ 
(cf. proof of Lemma \ref{fundlemma}). 
So, we have 
$$
u\ro g^m e_1\cdots e_{j-1}e_{j+1}\cdots e_{j+\frac{n}{2}-1}e_{j+\frac{n}{2}+1}\cdots e_n
\quad\text{and}\quad 
v\ro hg^{m'} e_1\cdots e_{j-1}e_{j+1}\cdots e_{j+\frac{n}{2}-1}e_{j+\frac{n}{2}+1}\cdots e_n
$$
and, by Lemma \ref{fundlemma}, $m=rn+j-i$, for some $r\in\{0,1\}$, and $m'=i+j-1-r'n$, for some $r'\in\{0,1\}$. 
Thus, we get 

$$
\begin{array}{rcll}
u & \rho & g^me_1\cdots e_{j-1}e_{j+1}\cdots e_{j+\frac{n}{2}-1}e_{j+\frac{n}{2}+1}\cdots e_n & \\
 & \rho & g^m hg^{2j-1}e_1\cdots e_{j-1}e_{j+1}\cdots e_{j+\frac{n}{2}-1}e_{j+\frac{n}{2}+1}\cdots e_n & \mbox{(by Lemma \ref{pre0})} \\ 
 & \rho & g^m hg^{2j-1+(r-r')n}e_1\cdots e_{j-1}e_{j+1}\cdots e_{j+\frac{n}{2}-1}e_{j+\frac{n}{2}+1}\cdots e_n & \mbox{(by $R_1$)} \\ 
 & \rho & hg^{n-m} g^{m+m'}e_1\cdots e_{j-1}e_{j+1}\cdots e_{j+\frac{n}{2}-1}e_{j+\frac{n}{2}+1}\cdots e_n & \mbox{(by $R_1$)} \\ 
 & \rho & hg^{m'}e_1\cdots e_{j-1}e_{j+1}\cdots e_{j+\frac{n}{2}-1}e_{j+\frac{n}{2}+1}\cdots e_n & \mbox{(by $R_1$)} \\ 
 & \rho & v. & 
\end{array}
$$

\smallskip 

Finally, consider that $k=n-1$. Let $i\in\Omega_n$ be such that $\Omega_n\setminus\{i_1,\ldots,i_{n-1}\}=\{i\}$. 
Then $\im(\alpha)=\{i\}$ and $\{i_1,\ldots,i_{n-1}\}=\{1,\ldots,i-1,i+1,\ldots,n\}$. 

Take $a=i\alpha^{-1}$. Then $ag^m=i=ahg^{m'}$. 
Since $ag^m=a+m-rn$, for some $r\in\{0,1\}$, and $ahg^{m'}=(n-a+1)g^{m'}=r'n-a+1+m'$, for some $r'\in\{0,1\}$, 
in a similar way to what we proved before, 
we have 
$$
\begin{array}{rcll}
u & \rho & g^me_1\cdots e_{i-1}e_{i+1}\cdots e_n & \\
 & \rho & g^m hg^{2i-1}e_1\cdots e_{i-1}e_{i+1}\cdots e_n & \mbox{(by Lemma \ref{pre1})} \\ 
 & \rho & g^m hg^{2i-1+(r-r')n}e_1\cdots e_{i-1}e_{i+1}\cdots e_n & \mbox{(by $R_1$)} \\ 
 & \rho & hg^{n-m} g^{m+m'}e_1\cdots e_{i-1}e_{i+1}\cdots e_n & \mbox{(by $R_1$)} \\ 
 & \rho & hg^{m'}e_1\cdots e_{i-1}e_{i+1}\cdots e_n & \mbox{(by $R_1$)} \\ 
 & \rho & v , & 
\end{array}
$$
as required.  
\end{proof}

Notice that, taking into account the relation $h^2=1$ of $R_1$, we could have taken only \textit{half} of the relations $R_5$, 
namely the relations $he_i=e_{n-i+1}h$ with $1\leqslant i\leqslant \lceil \frac{n}{2}\rceil$. 

\medskip

Our next and final goal is to deduce from the previous presentation for $\DPC_n$ a new one on $3$ generators, by using Tietze transformations. 

\smallskip 

Recall that, towards the end of Section \ref{basics}, we observed that $e_i=hg^{i-1}e_nhg^{i-1}$ for all $i\in\{1,2,\ldots,n\}$. 

\smallskip 

We will proceed as follows: first, by applying T1, we add the relations $e_i=hg^{i-1}e_nhg^{i-1}$, for $1\leqslant i\leqslant n$; 
secondly, we apply T4 to each of the relations $e_i=hg^{i-1}e_nhg^{i-1}$ with $i\in\{1,2,\ldots,n-1\}$ and, 
in some cases, by convenience, we also replace  $e_n$ by $hg^{n-1}e_nhg^{n-1}$; 
finally, by using the relations $R_1$, we simplify the new relations obtained, eliminating the trivial ones or those that are deduced from others. 
In what follows, we perform this procedure for each of the sets of relations $R_1$ to $R_6^\text{o}/R_6^\text{e}$. 
\begin{description}
\item $(R_1)$ 
There is nothing to do for these relations. 

\item $(R_2)$ 
For $1\leqslant i\leqslant n-1$, from $e_i^2=e_i$, we have 
$
hg^{i-1}e_nhg^{i-1} hg^{i-1}e_nhg^{i-1} = hg^{i-1}e_nhg^{i-1},
$ 
which is equivalent to $e_n^2=e_n$. 

\item $(R_3)$ 
For $1\leqslant i<j\leqslant n$, from $e_ie_j=e_je_i$, we get 
$
hg^{i-1}e_nhg^{i-1} hg^{j-1}e_nhg^{j-1} = hg^{j-1}e_nhg^{j-1} hg^{i-1}e_nhg^{i-1} 
$
and it is a routine matter to check that this relation is equivalent to $e_ng^{j-i}e_ng^{n-j+i} = g^{j-i}e_ng^{n-j+i}e_n$. 

\item $(R_4)$  
From $ge_1=e_ng$, we obtain 
$
ghe_nh=e_ng,
$ 
which is equivalent to $hg^{n-1}e_nhg^{n-1}=e_n$ (and, obviously, also to $ghe_ngh=e_n$). 
On the other hand, for $1\leqslant i\leqslant n-1$, 
from $ge_{i+1}=e_ig$ we get 
$
g hg^ie_nhg^i = hg^{i-1}e_nhg^{i-1} g 
$
and this relation is equivalent to $e_n=e_n$. 

\item $(R_5)$ 
For $1\leqslant i\leqslant n$, from $he_i=e_{n-i+1}h$, 
we have 
$
h hg^{i-1}e_nhg^{i-1} = hg^{n-i}e_nhg^{n-i} h,
$ 
which is a relation equivalent to $hg^{n-1}e_nhg^{n-1}=e_n$. 

\item $(R_6^\text{o})$  
From $hge_2e_3\cdots e_n=e_2e_3\cdots e_n$ (with $n$ odd), we get 
\begin{align*}
hg (hge_nhg) (hg^2e_nhg^2)\cdots  (hg^{n-1}e_nhg^{n-1}) = 
(hge_nhg) (hg^2e_nhg^2)\cdots  (hg^{n-1}e_nhg^{n-1}). 
\end{align*} 
It is easy to check that this relation is equivalent to $hg(e_ng)^{n-2}e_n=(e_ng)^{n-2}e_n$. 

\item $(R_6^\text{e})$ 
Now, for an even $n$, from $hge_2\cdots e_{\frac{n}{2}}e_{\frac{n}{2}+2}\cdots e_n=e_2\cdots e_{\frac{n}{2}}e_{\frac{n}{2}+2}\cdots e_n$, we obtain 
\begin{align*}
hg (hge_nhg) \cdots  (hg^{\frac{n}{2}-1}e_nhg^{\frac{n}{2}-1}) (hg^{\frac{n}{2}+1}e_nhg^{\frac{n}{2}+1})\cdots (hg^{n-1}e_nhg^{n-1}) = \\
(hge_nhg) \cdots  (hg^{\frac{n}{2}-1}e_nhg^{\frac{n}{2}-1}) (hg^{\frac{n}{2}+1}e_nhg^{\frac{n}{2}+1})\cdots (hg^{n-1}e_nhg^{n-1}),  
\end{align*}
which can routinely be verified to be equivalent to 
$hg (e_ng)^{\frac{n}{2}-1}g(e_ng)^{\frac{n}{2}-2}e_n= (e_ng)^{\frac{n}{2}-1}g(e_ng)^{\frac{n}{2}-2}e_n$. 
On the other hand, from $he_1e_2\cdots e_n=e_1e_2\cdots e_n$, we have 
\begin{align*}
h (he_nh) (hge_nhg)\cdots  (hg^{n-1}e_nhg^{n-1}) = 
(he_nh) (hge_nhg) \cdots  (hg^{n-1}e_nhg^{n-1}), 
\end{align*} 
a relation that is equivalent to $h(e_ng)^{n-1}e_n=(e_ng)^{n-1}e_n$. 
\end{description}

\smallskip 

So, let us consider the following set $Q$ of monoid relations on the alphabet $B=\{g,h,e\}$: 
\begin{description}
\item $(Q_1)$ $g^n=1$, $h^2=1$ and $hg=g^{n-1}h$; 

\item $(Q_2)$ $e^2=e$ and $ghegh=e$; 

\item $(Q_3)$ $eg^{j-i}eg^{n-j+i} = g^{j-i}eg^{n-j+i}e$, for $1\leqslant i<j\leqslant n$; 

\item $(Q_4)$ $hg(eg)^{n-2}e=(eg)^{n-2}e$, if $n$ is odd; 

\item $(Q_5)$ $hg (eg)^{\frac{n}{2}-1}g(eg)^{\frac{n}{2}-2}e= (eg)^{\frac{n}{2}-1}g(eg)^{\frac{n}{2}-2}e$ 
and $h(eg)^{n-1}e=(eg)^{n-1}e$, if $n$ is even. 
\end{description}
Notice that 
$|Q|=\frac{n^2-n+13+(-1)^n}{2}$. 

Therefore, by considering the mapping $B\longrightarrow\DPC_n$ defined by $g\longmapsto g$, $h\longmapsto h$ and $e\longmapsto e_n$, 
we have: 

\begin{theorem}\label{rankpres} 
The monoid $\DPC_n$ is defined by the presentation $\langle B \mid Q\rangle$ on $3$ generators. 
\end{theorem}

\bigskip 

\lastpage 

\end{document}